\documentclass{article}
\usepackage{lipsum}
\usepackage{amsfonts}
\usepackage{bm}
\usepackage{amsmath}
\newtheorem{theorem}{Theorem}[section]
\newtheorem{corollary}{Corollary}
\newtheorem{lemma}[theorem]{Lemma}
\newtheorem{proposition}{Proposition}

\newtheorem{remark}{Remark}
\newtheorem{assumption}[theorem]{Assumption}
\newtheorem{example}[theorem]{Example}
\newenvironment{proof}[1][Proof]{\noindent\textbf{#1.} }{\ \rule{0.5em}{0.5em}}

\begin{document}
\begin{center}
{\bf \huge Long-time existence of solutions to nonlocal nonlinear bidirectional wave equations}
\\ ~ \\ \vspace*{20pt}
H. A. Erbay$^{1}$, S. Erbay$^{1}$,  A. Erkip$^{2}$
\vspace*{20pt}

$^{1}$Department of Natural and Mathematical Sciences, Faculty of Engineering, Ozyegin University,  Cekmekoy 34794, Istanbul, Turkey
\vspace*{20pt}

$^{2}$Faculty of Engineering and Natural Sciences, Sabanci University,  Tuzla 34956,  Istanbul,    Turkey

\end{center}

\let\thefootnote\relax\footnote{E-mail:   husnuata.erbay@ozyegin.edu.tr, saadet.erbay@ozyegin.edu.tr, \\
albert@sabanciuniv.edu}

\begin{abstract}
We consider the Cauchy problem defined for a general class of nonlocal wave equations modeling bidirectional wave propagation in a nonlocally and nonlinearly elastic medium whose constitutive equation is given by a convolution integral.  We prove a long-time existence result for the nonlocal wave equations with a power-type nonlinearity and a small parameter.  As the energy estimates involve a loss of derivatives, we follow the Nash-Moser approach proposed by Alvarez-Samaniego and Lannes. As an application to the long-time existence theorem, we consider the limiting case in which the kernel function is the Dirac measure and the nonlocal equation reduces to the governing equation of one-dimensional classical elasticity theory. The present study also extends our earlier result concerning local well-posedness for smooth kernels to nonsmooth kernels.
\end{abstract}
\vspace*{10pt}

\vspace*{10pt}


\section{Introduction}\label{sec1}

In the present paper we prove the long-time existence and uniform estimates of solutions to the Cauchy problem
\begin{eqnarray}
   && u_{tt}=\beta\ast \big( u+\epsilon^{p}u^{p+1}\big)_{xx}, ~~~~ x\in \mathbb{R},~~~t>0,  \label{nw} \\
   && u(x,0)= u_{0}(x), ~~~~ u_{t}(x,0)= u_{1}(x)  \label{ini}
\end{eqnarray}
for sufficiently smooth initial data. Here  $u(x,t)$ is a real-valued function, $\epsilon$ is a small positive parameter  measuring the smallness of the initial data, the symbol $\ast$ denotes convolution in the $x$ -variable and $p$ is a positive integer. We assume that the convolution with the kernel function $\beta$ is a positive bounded operator on the Sobolev space $H^{s}(\mathbb{R})$. This can be realized by assuming that $\beta$ is integrable or more generally is a finite measure on $\mathbb{R}$ with positive Fourier transform.

The  nonlocal wave equation (\ref{nw}) describes the one-dimensional motion of a nonlocally and nonlinearly elastic medium and $u$ represents the elastic strain (we refer the reader to \cite{Duruk2010} both for a detailed description of the nonlocally and nonlinearly elastic medium and for some examples of the kernels widely used in continuum mechanics). Moreover, (\ref{nw}) involves many well-known nonlinear wave equations for particular choices of the kernel function. One well-known example is the improved Boussinesq  equation
\begin{equation}
    u_{tt}-u_{xx}-u_{xxtt}-\epsilon^{p}\big(u^{p+1}\big)_{xx}=0  \label{ib}
\end{equation}%
corresponding to the exponential kernel $\beta (x)=\frac{1}{2}e^{-\left\vert x\right\vert }$. On the other hand, if  $\beta$ is taken as the Dirac measure,  (\ref{nw}) reduces to the nondispersive nonlinear wave equation
 \begin{equation}
        u_{tt}-u_{xx}= \epsilon^{p}\big(u^{p+1}\big)_{xx}  \label{wa-eq}
    \end{equation}%
 of classical elasticity. The local well-posedness of (\ref{nw})-(\ref{ini})  was proved in \cite{Duruk2010} under a smoothness assumption on $\beta$. This assumption is equivalent  to saying that the operator $\beta^{\prime\prime}\ast (.)$ is bounded on $H^{s}(\mathbb{R})$.  In that case, (\ref{nw}) becomes a Banach space-valued ODE and the local existence result holds without any smallness assumption on the initial data.

In the present study, we  consider the long-time existence of solutions to (\ref{nw})-(\ref{ini}) and  provide an existence result on time intervals of order $1/\epsilon^{p}$. Additionally, we relax the restriction imposed in \cite{Duruk2010} on  the smoothness of $\beta$; in this case the smallness of $\epsilon$ guarantees that (\ref{nw}) stays in the hyperbolic regime. At this point, it is worth pointing out  that the smallness of the parameter $\epsilon$ plays an essential role in obtaining the long-time existence result. The key point in our approach is to prove that the bounds can be made uniform in  $\epsilon$.

To prove our long-time existence result we start by converting   (\ref{nw})  into a perturbation of the symmetric hyperbolic linear system and obtain the energy estimates  for the corresponding linearized equation. Nevertheless, the energy estimates involve a loss of derivatives. Due to the loss of derivatives one cannot directly pass from the linearized equation to the nonlinear equation via the Picard iteration, and we need  a Nash-Moser-type approach.    In that respect, the Nash-Moser theorem  proved in \cite{Alvarez2008} for a wide class of singular evolution equations is the main technical tool used in the present work.

In the literature, there are a number of studies concerning long-time existence of solutions to PDEs.  The long-time existence results of the studies focused on water waves \cite{Ming2012,Saut2012,Saut2017} have been also used for the rigorous justification of approximate asymptotical models (such as the Green-Naghdi equations, the shallow water equations and the Boussinesq system) starting from the Euler equations describing the motion of an inviscid, incompressible fluid. In addition to the studies about asymptotic models of water waves, there are also studies presenting the rigorous derivation of  various asymptotic models for nonlinear elastic waves in the long-wave-small-amplitude regime (for instance we refer the reader to \cite{Erbay2016} where  the Camassa-Holm equation and (\ref{nw}) are compared). The present research is motivated by the long-time existence results that were reported for water waves and aims to extend those results to elastic waves propagating in nonlocal elastic solids. In a previous work \cite{Erbay2016}, for smooth kernels and quadratic nonlinearity ($p=1$) we rigorously established  that, in the long-wave-small-amplitude regime, unidirectional solutions of (\ref{nw})  tend to associated solutions of Korteweg-de Vries-type, or Benjamin-Bona-Mahony-type or Camassa-Holm-type equations depending on  the balance between dispersion and nonlinearity. Providing a precise control of the approximation error through the justification process, we showed that those unidirectional asymptotic models are good approximations for (\ref{nw}) on time intervals of order $1/\epsilon$. There we used the fact that solutions of the unidirectional asymptotic models live over a long-time scale \cite{Constantin2009}.  Hence the corresponding solution of the nonlocal wave equation will live on a sufficiently large  time interval. As future work, we plan to investigate  similar comparison results between two nonlocal equations. The long-time existence result and uniform bound obtained in this study will be used in a future work to explore comparison of nonlocal equations.

The structure of the paper is as follows. Section \ref{sec2} is devoted to preliminaries, where  (\ref{nw}) is converted into a system of first-order equations, and  some function spaces are introduced. In Section \ref{sec3}, we consider a related linear system and derive the estimates to be used in the next section. In Section \ref{sec4} we consider the Nash-Moser approach of \cite{Alvarez2008} and  prove that the required assumptions in the hypothesis of the existence theorem in \cite{Alvarez2008} are satisfied in our case. Finally, in the last section we present our long-time existence result and discuss some particular cases.

Throughout this paper, we use the standard notation for function spaces. The Fourier transform $\widehat u$ of $u$ is defined by $\widehat u(\xi)=\int_\mathbb{R} u(x) e^{-i\xi x}dx$. We  also use  $\mathcal{F}$ and  $\mathcal{F}^{-1}$ to denote the Fourier transform and the inverse Fourier transform.  The $L^p(\mathbb{R})$  norm  is denoted by $\Vert u\Vert_{L^p}$ and the symbol $\langle u, v\rangle_{L^{2}}$ denotes the inner product of $u$ and $v$ in $L^2$.  $H^{s}=H^s(\mathbb{R})$  denotes the $L^{2}$-based Sobolev space of order $s$ on $\mathbb{R}$, with the norm $\Vert u\Vert_{H^{s}}=\left(\int_\mathbb{R} (1+\xi^2)^s \vert\widehat u(\xi)\vert^2 d\xi \right)^{1/2}$.  $C$ is a generic positive constant. Partial differentiations are denoted by $D_{x}$ etc.

\section{Preliminaries}\label{sec2}

In this section, we recast (\ref{nw}) as an appropriate first-order system, recall the local existence theorem from \cite{Duruk2010} and introduce function spaces that will be used in the paper. For the rest of this work, we assume that the kernel is an integrable function (or more generally a finite measure) satisfying the nonnegativity and boundedness condition
\begin{equation}
0\leq \widehat{\beta}(\xi)\leq C. \label{betabound}
\end{equation}

We first convert the Cauchy problem (\ref{nw})-(\ref{ini}) to
 \begin{eqnarray}
    && u_{t}=K v_{x}, \hspace*{78pt} u(x,0)=u_{0}(x), \label{systema} \\
    && v_{t} =K u_x + \epsilon^p K\big(u^{p+1}\big)_x, ~~~~ v(x,0)=v_{0}(x) \label{systemb}
\end{eqnarray}
 by writing the nonlocal equation (\ref{nw}) as a first-order system and introducing the pseudo-differential operator
\begin{equation}
    K w(x)=\mathcal{F}^{-1}\Big( \sqrt{ \widehat{\beta }(\xi )} \widehat{w}(\xi )\Big),  \label{K1}
\end{equation}
for which  $K^2 w=\beta*w$. We note that, by (\ref{betabound}), $K$ is a bounded operator on $H^{s}$. Clearly, for the choice $u_{1}=K(v_{0})_{x}$ the initial-value problem (\ref{nw})-(\ref{ini})  reduces to the first-order system (\ref{systema})-(\ref{systemb}).

 The local well-posedness of the Cauchy problem (\ref{nw})-(\ref{ini}), equivalently the local-well posedness of   (\ref{systema})-(\ref{systemb}), was proved in \cite{Duruk2010} under the  regularity assumption
\begin{equation}
    0\leq \widehat{\beta }(\xi )\leq C (1+\xi^2)^{-r/2},~\quad r\geq 2. \label{order}
\end{equation}
In such a case, $K$ is an operator of order $-r/2$ and hence maps $H^{s}$ into $H^{s+\frac{r}{2}}\subset H^{s}$. For $r\geq 2$, due to the regularizing effect, both (\ref{nw})-(\ref{ini}) and (\ref{systema})-(\ref{systemb}) are the initial-value problems defined for $H^{s}$-valued ODE's. Consequently, the  Cauchy problems are locally well-posed with solutions in $C^{1}\big([0,T],H^{s}\big)$ for some $T>0$. Obviously, the parameter $r$ in (\ref{order}) is a measure of the smoothness of $\beta$ and hence the regularizing effect of the convolution operator. Namely,   the regularizing effect  increases as the decay rate $r$ gets larger. In \cite{Duruk2010} it was shown that, when $r\geq2$, a possible finite-time blow-up of solutions will be controlled by the $L^{\infty}$-norm.  On the other hand, when the smoothing effect of $\beta$ is weaker,  solutions may behave in a nonlinear hyperbolic manner and  may evolve to breaking.

In the present work we extend the local well-posedness result  proved for the case $r\geq 2$ in \cite{Duruk2010} to the case $r\geq 0$ and show long-time existence in both cases. Indeed, the existence of two different intervals for $r$ can be attributed to the dual nature of (\ref{nw}). Due to the existence of second-order spatial derivative in  (\ref{nw}), $r=2$  is the threshold value determining whether (\ref{nw}) behaves  like  a hyperbolic equation or an ODE.  Here we merely state that the well-posedness question of (\ref{nw})-(\ref{ini}) for the case $0 \leq r <2$ requires the techniques that are different from those in \cite{Duruk2010}. To prove the long-time existence of the solution we need  suitable energy estimates. In our problem, even for the regularized case, the energy estimates involve a loss of derivatives. To overcome the loss of derivatives we use a Nash-Moser-type approach. This   enables us to extend the existence result to the hyperbolic regime. In particular, we follow the approach given in \cite{Alvarez2008} and prove that a unique solution to (\ref{nw})-(\ref{ini}) exists over long-time scales of order $1/\epsilon^p$ in appropriate function spaces.

We now define the function spaces  that will be used  in  this work. For a fixed time $T$ let
\begin{displaymath}
        H_{\epsilon (0)}^{s} =C\big([0,{\frac{T}{\epsilon ^{p}}}];H^{s}\big),\quad H_{\epsilon (1)}^{s}=C\big([0,{\frac{T}{\epsilon ^{p}}}];H^{s}\big)\cap C^{1}\big([0,{\frac{T}{\epsilon ^{p}}}];H^{s-1}\big),
\end{displaymath}%
    with norms
\begin{displaymath}
        \Vert u\Vert _{H_{\epsilon(0)}^{s}} =\sup_{t\in \lbrack 0,{\frac{T}{\epsilon ^{p}}}]}\Vert u(t)\Vert _{H^{s}},~~~~ \Vert u\Vert _{H_{\epsilon(1)}^{s}}=\sup_{t\in \lbrack 0,{\frac{T}{\epsilon ^{p}}}]}\big(\Vert u(t)\Vert _{H^{s}}+\Vert u_{t}(t)\Vert _{H^{s-1}}\big).
\end{displaymath}
We  also introduce the vector-valued counterparts
\begin{eqnarray*}
       && X_{\epsilon (0)}^{s}=C\big([0,{\frac{T}{\epsilon ^{p}}}];X^{s}\big), \hspace*{40pt}
        X_{\epsilon (1)}^{s}=C\big([0,{\frac{T}{\epsilon ^{p}}}];X^{s}\big)\cap C^{1}\big([0,{\frac{T}{\epsilon ^{p}}}];X^{s-1}\big), \\
       && \Vert \mathbf{u}\Vert _{X_{\epsilon(0)}^{s}}=\sup_{t\in [0,{\frac{T}{\epsilon ^{p}}}]}\Vert \mathbf{u}(t)\Vert _{X^{s}},           ~~~~~
        \Vert \mathbf{u}\Vert _{X_{\epsilon(1)}^{s}} =\sup_{t\in [0,{\frac{T}{\epsilon ^{p}}}]}\big(\Vert \mathbf{u}(t)\Vert _{X^{s}}+\Vert \mathbf{u}_{t}(t)\Vert _{X^{s-1}}\big),
\end{eqnarray*}
where
\begin{displaymath}
        X^{s} =H^{s}\times H^{s},~~~~~ \Vert \mathbf{u}\Vert _{X^{s}} =\Vert (u,v)\Vert _{X^{s}}=\Vert u\Vert _{H^{s}}+\Vert v\Vert _{H^{s}}.
\end{displaymath}
Note that these spaces depend on two parameters $T$ and $\epsilon$, but, to simplify the notation, we have suppressed the index $T$.

\section{Energy estimates for a related linear system}\label{sec3}

As a starting point, we consider  the following initial-value problem
\begin{eqnarray}
    && u_{t}=Kv_{x}+\epsilon ^{p}f_{1},\hspace*{69pt} u(x,0)=g_{1}(x), \label{l1} \\
    && v_{t}=Ku_{x}+\epsilon ^{p}K(wu)_{x}+\epsilon ^{p}f_{2},~~~~  v(x,0)=g_{2}(x), \label{l2}
\end{eqnarray}
defined for a nonhomogeneous linear system of differential equations associated with (\ref{systema})-(\ref{systemb}),  where $w$, $f_{i}$ and $g_{i}$ ($i=1,2$) are fixed, given and sufficiently smooth functions. In this section we obtain  a priori estimates for solutions of (\ref{l1})-(\ref{l2}) over the long-time intervals  $[0,{\frac{T}{\epsilon ^{p}}}]$.

We begin by defining the energy functional for the linear system (\ref{l1})-(\ref{l2}) as follows
\begin{equation}
    E_{s}^{2}(t)=\frac{1}{2}\big( \Vert u(t)\Vert_{H^{s}}^{2}+\Vert v(t)\Vert_{H^{s}}^{2}+\epsilon ^{p}\langle u(t),w(t)u(t)\rangle_{H^{s}}\big)  \label{ener}
\end{equation}%
where we have made use of the representation  $\langle u,v\rangle _{H^{s}}=\langle \Lambda ^{s}u,\Lambda^{s}v\rangle _{L^{2}}$ with $\Lambda ^{2s}=(1-D_{x}^{2})^{s}$. The following lemma states that, under some smallness assumption on $\epsilon $, the energy functional  $E_{s}$ is uniformly equivalent to the $X^{s}$ norm of the vector $\mathbf{u}=(u,v)$.
\begin{lemma}\label{lem3.1}
    Let $w\in H_{\epsilon(0)}^{s}$. Then there is some $\epsilon _{0}>0$ so that, for all $0<\epsilon \leq \epsilon _{0}$, $~ E_{s}(t)$ is uniformly equivalent to $\Vert \mathbf{u}(t)\Vert _{X^{s}}$.
\end{lemma}
\begin{proof}
    Recall that, for $s>1/2$, $H^{s}$ is an algebra:
    \begin{displaymath}
    |\langle u(t),w(t)u(t)\rangle _{H^{s}}|\leq C\Vert u(t)\Vert_{H^{s}}^{2}\Vert w(t)\Vert _{H^{s}}.
    \end{displaymath}
     For
    \begin{displaymath}
        \epsilon ^{p}\leq \frac{1}{2C\Vert w\Vert _{H_{\epsilon (0)}^{s}}}=\epsilon _{0}^{p},
    \end{displaymath}%
    we have
    \begin{displaymath}
        -\frac{1}{2}\Vert u(t)\Vert _{H^{s}}^{2}\leq \epsilon ^{p}\langle u(t),w(t)u(t)\rangle _{H^{s}}
            \leq \frac{1}{2}\Vert u(t)\Vert _{H^{s}}^{2}.
    \end{displaymath}%
    Using  this and (\ref{ener}) we get
    \begin{displaymath}
        \frac{1}{2}\big(\Vert u(t)\Vert _{H^{s}}^{2}+\Vert v(t)\Vert _{H^{s}}^{2}-\frac{1}{2}\Vert u(t)\Vert _{H^{s}}^{2}\big)
        \leq E_{s}^{2}(t)
        \leq \frac{1}{2}\big(\Vert u(t)\Vert _{H^{s}}^{2}+\Vert v(t)\Vert _{H^{s}}^{2}+\frac{1}{2}\Vert u(t)\Vert _{H^{s}}^{2}\big)
    \end{displaymath}
    from which it follows that
     \begin{equation}
        {1\over 2\sqrt{2}}\big(\Vert u(t)\Vert _{H^{s}}+\Vert v(t)\Vert _{H^{s}}\big)
        \leq E_{s}(t)
        \leq {\sqrt{3}\over 2}\big(\Vert u(t)\Vert _{H^{s}}+\Vert v(t)\Vert _{H^{s}}\big). \label{lin-en}
    \end{equation}
\end{proof}
\begin{remark}
    Lemma \ref{lem3.1} shows that   if $\epsilon < \epsilon_{0}$, then $\epsilon^{p} \Vert w(t)\Vert _{H^{s}}$ is small. This  in turn implies the usual hyperbolicity condition $1+\epsilon^{p} w>0$ for (\ref{l1})-(\ref{l2}).
\end{remark}

Next we state a lemma about the product and commutator estimates that will be needed in the sequel. It corresponds to Lemma 4.6 of \cite{Alvarez2008} and provides a classical Moser tame product estimate and a particular case of Kato-Ponce commutator estimate:
\begin{lemma}
    Let $s>s_0>1/2$.
    \begin{enumerate}
        \item  For all $f, g\in H^s(\mathbb{R})$, one has
            \begin{equation}
                \Vert fg\Vert_{H^s} \leq  C \big( \Vert f\Vert_{H^{s_0}} \Vert g\Vert_{H^s}
                + \Vert f\Vert_{H^s}\Vert g\Vert_{H^{s_0}}\big).  \label{moser}
            \end{equation}
        \item  Let $r\in \mathbb{R}$ be such that $~ -s_{0}<r\leq s_{0}+1$. For all $f\in H^{s_{0}+1}(\mathbb{R})\cap H^{s+r}(\mathbb{R})$ and $u\in H^{s+r-1} (\mathbb{R})$,
            \begin{equation}
               \big \Vert [\Lambda^s, f]u\big\Vert_{H^r} \leq C\big(\Vert f_x\Vert_{H^{s_0}} \Vert u\Vert_{H^{s+r-1}} + \Vert f_x\Vert_{H^{s+r-1}}\Vert u\Vert_{H^{s_0}}\big).  \label{kato}
            \end{equation}
        \end{enumerate}
\end{lemma}
In the rest of this work we will always assume that $s_{0}>\frac{1}{2}$.

We are now ready to state and prove an a priori energy estimate for the linear system (\ref{l1})-(\ref{l2}).
\begin{proposition}\label{prop3.3}
    Let $s\geq s_{0}+1$, $~T>0$, $~w\in H_{\epsilon(1)}^{s+1}$, $\mathbf{f}=(f_{1},f_{2})\in X_{\epsilon (0)}^{s}$, $\mathbf{g}=(g_{1},g_{2})\in X^{s}$. Suppose $\mathbf{u}=(u,v)\in X_{\epsilon (0)}^{s}$ satisfies the initial-value problem (\ref{l1})-(\ref{l2}). Then, there is some $\epsilon_{0}$ such that  for all $0<\epsilon < \epsilon_{0}$ and $t\in [0, {T\over \epsilon^{p}}]$
    \begin{equation}
        \Vert\mathbf{u}(t)\Vert_{X^s}
            \leq C\big(T,\Vert w\Vert_{H_{\epsilon(1)}^{s_{0}+2}}\big)  \Big( \mathfrak{I}^s(t,\mathbf{f},\mathbf{g})  +\Vert w\Vert_{H_{\epsilon(1)}^{s+1}}  \mathfrak{I}^{s_{0}+1}(t,\mathbf{f},\mathbf{g})\Big), \label{prop33}
    \end{equation}
    where $\epsilon_{0}$ is determined as in Lemma \ref{lem3.1} and
    \begin{equation}
        \mathfrak{I}^{s}(t,\mathbf{f},\mathbf{g})
        =\Vert \mathbf{g}\Vert_{X^{s}}+\int_{0}^{t}\sup_{0\leq t^{\prime} \leq t^{\prime\prime}}\Vert \mathbf{f}(t^{\prime})\Vert _{X^{s}}dt^{\prime\prime} . \label{prodef}
    \end{equation}
\end{proposition}
\begin{proof}
    Taking the $L^{2}$ inner product of (\ref{l1}) with $\Lambda^{2s}u$, (\ref{l2}) with $\Lambda ^{2s}v$ and adding them up yields
    \begin{eqnarray}
        {\frac{d}{dt}}E_{s}^{2}(t)
        &=&\epsilon ^{p}\Big (\langle u,f_{1}\rangle _{H^{s}}+\langle v,f_{2}\rangle _{H^{s}}\Big)
            +\epsilon^{p}\langle v,K(wu)_{x}\rangle_{H^{s}}
            +{\frac{\epsilon ^{p}}{2}}{\frac{d}{dt}}\langle u,wu\rangle _{H^{s}}  \nonumber  \\
        &=&\epsilon ^{p}\Big (\langle u,f_{1}\rangle _{H^{s}}+\langle v,f_{2}\rangle_{H^{s}}\Big)
            -\epsilon ^{p}\langle Kv_{x},wu\rangle _{H^{s}}  \nonumber \\
        && +{\frac{\epsilon ^{p}}{2}}\Big (\langle u_{t},wu\rangle _{H^{s}}+\langle u,w_{t}u\rangle _{H^{s}}
            +\langle u,wu_{t}\rangle _{H^{s}}\Big )  \nonumber \\
        &=&\epsilon^{p}\Big (\langle u,f_{1}\rangle _{H^{s}}+\langle v,f_{2}\rangle_{H^{s}}\Big)
            +\epsilon ^{2p}\langle f_{1},wu\rangle _{H^{s}}
            +{\frac{\epsilon^{p}}{2}}\langle u,w_{t}u\rangle _{H^{s}}  \nonumber \\
        && +{\frac{\epsilon ^{p}}{2}}\Big (\langle u,wu_{t}\rangle _{H^{s}}
            -\langle u_{t},wu\rangle _{H^{s}}\Big ),  \label{en}
    \end{eqnarray}%
    where we have used (\ref{l1}) and (\ref{ener}). We  now estimate the terms in (\ref{en}). The first and second terms on the right-hand side of (\ref{en}) are estimated as
    \begin{equation}
        \big| \langle u,f_{1}\rangle _{H^{s}}+\langle v,f_{2}\rangle_{H^{s}}\big|
            \leq C\big(\Vert f_{1}\Vert _{H^{s}}\Vert u\Vert _{H^{s}}+\Vert f_{2}\Vert _{H^{s}}\Vert v\Vert _{H^{s}}\big), \label{term1}
    \end{equation}
    and
    \begin{equation}
        \big| \langle f_{1},wu\rangle _{H^{s}}\big|
            \leq C\Vert f_{1}\Vert _{H^{s}}\Vert w\Vert _{H^{s}}\Vert u\Vert _{H^{s}}, \label{term2}
    \end{equation}
    respectively.     Using the Cauchy-Schwarz inequality and the  estimate (\ref{moser}), we get
    \begin{equation}
        \big| \langle u,w_{t}u\rangle _{H^{s}}\big|
        \leq C\big (\Vert u\Vert _{H^{s}}^{2}\Vert w_{t}\Vert _{H^{s_{0}}}
            +\Vert u\Vert _{H^{s}}\Vert u\Vert _{H^{s_{0}}}\Vert w_{t}\Vert _{H^{s}}\big)  \label{term3}
    \end{equation}%
    for the third term on the right-hand side of (\ref{en}).     Similarly, the use of  the Cauchy-Schwarz inequality and the commutator   $[\Lambda ^{s},w]f=\Lambda ^{s}(wf)-w\Lambda ^{s}f$ makes possible to write the last term  in (\ref{en}) as
    \begin{equation}
       \big| \langle u,wu_{t}\rangle _{H^{s}}-\langle u_{t},wu\rangle _{H^{s}}\big|
        \leq  C \big(\Vert \Lambda^{s}u\Vert_{L^{2}}\Vert \lbrack \Lambda^{s},w\rbrack u_{t}\Vert_{L^{2}}
            +\Vert \Lambda^{s-1}u_{t}\Vert_{L^{2}}\Vert\Lambda \lbrack \Lambda^{s},w \rbrack u\Vert_{L^{2}}\big). \label{last}
    \end{equation}
    Applying the Kato-Ponce commutator estimate (\ref{kato}) to the terms $\Vert \Lambda \lbrack \Lambda ^{s},w]u\Vert_{L^{2}}$ and  $\Vert \lbrack \Lambda^{s},w]u_{t}\Vert _{L^{2}}$  in (\ref{last}), we get
    \begin{eqnarray}
    \Vert \Lambda \lbrack \Lambda^{s},w]u\Vert_{L^{2}}
            &\leq & C \big (\Vert u\Vert_{H^{s}}\Vert w\Vert_{H^{s_{0}+1}}
            +\Vert u\Vert_{H^{s_{0}}}\Vert w_{x}\Vert_{H^{s}}\big), \label{fff} \\
    \Vert \lbrack \Lambda^{s},w]u_{t}\Vert _{L^{2}}
            &\leq &  C \big (\Vert u_{t}\Vert_{H^{s-1}}\Vert w_{x}\Vert_{H^{s_{0}}}
            +\Vert u_{t}\Vert_{H^{s_{0}}}\Vert w_{x}\Vert _{H^{s-1}}\big).    \label{sss}
    \end{eqnarray}
    Substituting  (\ref{fff}) and (\ref{sss}) into (\ref{last}) and using (\ref{l1}) to eliminate $u_{t}$ in the resulting expression we obtain
    \begin{eqnarray}
       \big| \langle u,w u_{t}\rangle _{H^{s}}-\langle u_{t},wu\rangle _{H^{s}} \big|
        &\leq & C \Big (\Vert u\Vert _{H^{s}}\Vert v\Vert _{H^{s}}\Vert w\Vert _{H^{s_{0}+1}}
            +\Vert u\Vert _{H^{s}}\Vert v\Vert _{H^{s_{0}+1}}\Vert w\Vert _{H^{s}} \nonumber   \\
        && + \Vert u\Vert _{H^{s_{0}}} \Vert v\Vert_{H^{s}}\Vert w\Vert _{H^{s+1}}
            +\epsilon^{p}\big (\Vert f_{1}\Vert_{H^{s_{0}}} \Vert u\Vert_{H^{s}}  \Vert w\Vert _{H^{s}} \nonumber  \\
        && +\Vert f_{1}\Vert_{H^{s}} \Vert u\Vert_{H^{s}}\Vert w\Vert_{H^{s_{0}+1}}
            +\Vert f_{1}\Vert_{H^{s}}\Vert u\Vert _{H^{s_{0}}}\Vert w\Vert_{H^{s}}\big )\Big). \nonumber \\ \label{term4}
    \end{eqnarray}
    Using (\ref{lin-en}) and the estimates (\ref{term1}), (\ref{term2}), (\ref{term3}), (\ref{term4}) in (\ref{en}) we obtain
    \begin{eqnarray}
        {d\over dt}E_s(t)
        &\leq &  C \epsilon^p \Big(\big(\Vert w_t\Vert_{H^{s_{0}}}+ \Vert w_x\Vert_{H^{s_{0}}}\big) E_s(t)
            +\Vert f_1\Vert_{H^s}+\Vert f_2\Vert_{H^s}
             \nonumber \\
        & ~& ~~~~~~~~~ +\epsilon^p \Vert f_1\Vert_{H^s} \big(\Vert w\Vert_{H^s} + \Vert w\Vert_{H^{s_{0}+1}}\big)\nonumber \\
        & ~&    +\big(\Vert w_x\Vert_{H^s}+\Vert w_t\Vert_{H^s}\big)\big(\Vert u\Vert_{H^{s_{0}}}
            + \Vert v\Vert_{H^{s_{0}+1}}\big) \Big )  \nonumber \\
        &\leq &C\epsilon^p \Big( \Vert w\Vert_{H_{\epsilon(1)}^{s_{0}+1}} E_s(t)
            +\Vert w\Vert_{H_{\epsilon(1)}^{s+1}} E_{s_{0}+1}(t) +(1+\epsilon^p \Vert w\Vert_{H_{\epsilon (0)}^s}) \Vert\mathbf{f}(t)\Vert_{X^s}\Big ) \nonumber  \\
        &\leq &C\epsilon^p \Big( \Vert w\Vert_{H_{\epsilon(1)}^{s_{0}+1}} E_s(t)
            +\Vert w\Vert_{H_{\epsilon(1)}^{s+1}} E_{s_{0}+1}(t) + \Vert\mathbf{f}(t)\Vert_{X^s}\Big),   \label{en1}
    \end{eqnarray}
    where we have used $\epsilon^p \Vert w\Vert_{H_{\epsilon (0)}^s} \leq C$.      Applying the Gronwall inequality,  $E_{s}(t)$  is estimated  as follows
    \begin{equation}
        E_s(t)
        \leq ~e^{C\epsilon^p t\Vert w\Vert_{H_{\epsilon(1)}^{s_{0}+1}}}\Big ( E_s(0) +C \epsilon ^p\Vert w\Vert_{H_{\epsilon(1)}^{s+1}}\int_0^t E_{s_{0}+1}(t^{\prime})dt^{\prime} +C\epsilon^{p} \int_0^t\Vert \mathbf{f}(t^{\prime})\Vert_{X^s} dt^{\prime} \Big ).  \label{es}
    \end{equation}
    The next step is to eliminate the term $E_{s_{0}+1}(t)$ in (\ref{es}). This is accomplished by getting a similar inequality for $E_{s_{0}+1}(t)$. For  $s=s_{0}+1$, the differential inequality (\ref{en1}) takes the form
    \begin{displaymath}
        {d \over dt}E_{s_{0}+1}(t)
        \leq C\epsilon ^p \Big (\Vert w\Vert_{H_{\epsilon (1)}^{s_{0}+2}}E_{s_{0}+1}(t) +\Vert \mathbf{f}(t)\Vert_{X^{s_{0}+1}} \Big ).
    \end{displaymath}
    Again, by the Gronwall inequality, we get
    \begin{displaymath}
        E_{s_{0}+1}(t)
        \leq ~e^{C\epsilon ^{p}t\Vert w\Vert _{H_{\epsilon(1)}^{s_{0}+2}}} \Big (E_{s_{0}+1}(0)+C\epsilon^p  \int_0^t \Vert\mathbf{f}(t^{\prime})\Vert_{X^{s_{0}+1}}dt^{\prime}\Big ).
    \end{displaymath}
    We need an estimate for time integral of $E_{s_{0}+1}(t)$
    \begin{eqnarray*}
    \int_0^t E_{s_{0}+1}(t^{\prime})dt^{\prime}
    &\leq &  \int_0^t e^{C\epsilon^{p}t^{\prime}\Vert w\Vert _{H_{\epsilon(1)}^{s_{0}+2}}} \Big (E_{s_{0}+1}(0)  \\
    && +C\epsilon^p \int_0^{t^{\prime}} \sup_{0\leq t^{\prime\prime} \leq t^{\prime\prime\prime}}\Vert\mathbf{f}(t^{\prime\prime})\Vert_{X^{s_0+1}}dt^{\prime\prime\prime}\Big ) dt^{\prime}\\
    &\leq & Cte^{C\epsilon^p t \Vert w\Vert _{H_{\epsilon(1)}^{s_{0}+2}}} \Big( E_{s_{0}+1}(0)
    + \int_0^t \sup_{0\leq t^{\prime} \leq t^{\prime\prime}}\Vert\mathbf{f}(t^{\prime})\Vert_{X^{s_0+1}}dt^{\prime\prime} \Big). \label{e2}
    \end{eqnarray*}
    Using this result in (\ref{es}) we obtain
    \begin{eqnarray}
        E_s(t)
        &\leq &~C e^{C\epsilon^p t\Vert w\Vert_{H_{\epsilon(1)}^{s_{0}+1}}} \bigg( E_s(0) + \int_0^t\sup_{0\leq t^{\prime}\leq t^{\prime\prime}} \Vert\mathbf{f}(t^{\prime})\Vert_{X^s} dt^{\prime\prime} \nonumber\\
        &&  + \epsilon^{p}t\Vert w\Vert_{H_{\epsilon(1)}^{s+1}}  e^{C\epsilon^p t\Vert w\Vert_{H_{\epsilon(1)}^{s_{0}+2}}} \Big( E_{s_{0}+1}(0) + \int_0^t \sup_{0\leq t^{\prime}\leq t^{\prime\prime}} \Vert\mathbf{f}(t^{\prime})\Vert_{X^{s_{0}+1}} dt^{\prime\prime} \Big)\bigg). \nonumber
    \end{eqnarray}
    Then, using the definition in (\ref{prodef}), we get
    \begin{displaymath}
        \Vert\mathbf{u}(t)\Vert_{X^s}
        \leq     C\big(T,\Vert w\Vert_{H_{\epsilon(1)}^{s_{0}+2}}\big) \Big (\mathfrak{I}^s(t,\mathbf{f}, \mathbf{g})    +  \Vert w\Vert_{H_{\epsilon(1)}^{s+1}} \mathfrak{I}^{s_{0}+1}(t,\mathbf{f}, \mathbf{g}) \Big )
    \end{displaymath}
    for $t\in[0,{\frac{T}{\epsilon ^{p}}}]$.     This completes the proof.
\end{proof}

For convenient reference in the remainder of the paper, it is useful to write our energy estimate in terms of a new scaled time variable instead of $t$. Let us do so by introducing the scaled time $\tau=\epsilon^{p}t$, so that we change the time interval  $[0, \frac{T}{\epsilon ^{p}}]$ for $t$ to $\left[ 0,T\right] $ for $\tau$. Then, the linear initial-value problem (\ref{l1})-(\ref{l2}) becomes
        \begin{eqnarray}
            && u_{\tau }=\frac{1}{\epsilon ^{p}}Kv_{x}+f_{1},\hspace*{58pt} u(x,0)=g_{1}(x)\label{l1tau} \\
            && v_{\tau }=\frac{1}{\epsilon ^{p}}Ku_{x}+K(wu)_{x}+f_{2},~~~v(x,0)=g_{2}(x)  \label{l2tau}
        \end{eqnarray}%
in the scaled variable $\tau$. In such case, all of the previous arguments used to prove the energy estimate of Proposition \ref{prop3.3} still hold for the new system with obvious modifications.  Indeed,   we now define the spaces $X_{(0)}^{s}$ and $X_{(1)}^{s}$ as $X_{(0)}^{s}=C\big([0,{T}];X^{s}\big)$ and $X_{(1)}^{s}=C\big([0,{T}];X^{s}\big)\cap  C^{1}\big([0,{T}];X^{s-1}\big)$, respectively. Additionally, the $X_{(1)}^{s}$-norm takes the form
\begin{displaymath}
        \Vert \mathbf{u}\Vert _{X_{(1)}^{s}}
        =\sup_{\tau \in \lbrack 0,{T}]}\Big(\Vert \mathbf{u}(\tau )\Vert _{X^{s}}
            +\epsilon ^{p}\Vert \mathbf{u}_{\tau}(\tau )\Vert _{X^{s-1}}\Big).
\end{displaymath}

\begin{corollary}\label{cor3.4}
    Let $s\geq s_{0}+1$, $T>0$, $w\in H_{(1)}^{s+1}$, $\mathbf{f}=(f_{1},f_{2})\in X_{(0)}^{s}$, $\mathbf{g}=(g_{1},g_{2})\in X^{s}$. Suppose $\mathbf{u}=(u,v)\in X_{ (0)}^{s}$ satisfies the initial-value problem (\ref{l1tau})-(\ref{l2tau}). Then, there is some $\epsilon_{0}$ such that  for all $0<\epsilon \leq \epsilon_{0}$ and $\tau \in [0, T]$
    \begin{displaymath}
        \Vert\mathbf{u}(\tau)\Vert_{X^s}
            \leq C\big(T,\Vert w\Vert_{H_{(1)}^{s_{0}+2}}\big)  \Big ( \mathfrak{I}^s(\tau,\mathbf{f},\mathbf{g})  +\Vert w\Vert_{H_{(1)}^{s+1}}  \mathfrak{I}^{s_{0}+1}(\tau,\mathbf{f},\mathbf{g})\Big  ).
    \end{displaymath}
\end{corollary}
We note that the constant $C$ in Proposition \ref{prop3.3} and Corollary \ref{cor3.4} also depends on the operator norm of $K$.

As seen in Corollary \ref{cor3.4}, there is a loss of derivative in the energy estimate for the linear system, that is, $\Vert \mathbf{u}(\tau)\Vert_{X^{s}}$ is controlled by $\Vert w\Vert_{H^{s+1}_{(1)}}$, the norm of the reference state. This loss of derivative propagates along the iteration scheme and may cause problems in a standard Picard iteration scheme for the nonlinear system (\ref{systema})-(\ref{systemb}). In order to handle the loss of derivative, we will use the Nash-Moser-type  approach  described in \cite{Alvarez2008} for a general system of evolution equations. The following section serves as preparation for the Nash-Moser scheme.

\section{Preparation for the Nash-Moser scheme}\label{sec4}

In preparation for the proof of our main result, we outline in this section certain preliminaries essential for understanding how the approach used in \cite{Alvarez2008} is related to the present case.

In \cite{Alvarez2008}, the authors have studied the well-posedness of the following general class of initial-value problems (see equation (1.1) of  \cite{Alvarez2008})
\begin{equation}
    \partial _{t}\mathbf{u}^{\epsilon }+\frac{1}{\epsilon }\mathcal{L}^{\epsilon}(t)\mathbf{u}^{\epsilon} +\mathcal{N}^{\epsilon }[t,\mathbf{u}^{\epsilon }] =\mathbf{h}^{\epsilon }, ~~~~
    \mathbf{u}^{\epsilon }(0) =\mathbf{u}_{0}^{\epsilon },  \label{alv}
\end{equation}%
where $\epsilon >0$ is a small parameter and $\mathcal{L}^{\epsilon }(t)$ and $\mathcal{N}^{\epsilon }[t,.]$ are  linear and nonlinear operators, respectively. By making three  simplifying assumptions, they have proved their  well-posedness theorem for time intervals $[0,\overline{T}]$ where $\overline{T}>0$ is independent of $\epsilon$. While two of the assumptions are concerned with $\mathcal{L}^{\epsilon }$ and $\mathcal{N}^{\epsilon }$, the third assumption is about the existence of a tame estimate for the solution of the related linearized system. Their main result  is the following general theorem:
\begin{theorem}\label{theo4.1}
    (Theorem 2.1 of \cite{Alvarez2008})
    Let $T>0,$ $\ s_{0},m,d_{1}$ and $d_{1}^{\prime }$ be such that Assumptions 1.2, 1.3 and 1.5 of \cite{Alvarez2008} are satisfied. Let also $D>\delta$, ~$P>P_{\min },~s\geq s_0+m$ and $(\mathbf{h}^{\epsilon },\mathbf{u}_{0}^{\epsilon })_{0<\epsilon <\epsilon_0}$ be bounded in $F^{s+P}.$ Then there exist $0<\underline{T}\leq T$ and a unique family $(\mathbf{u}^{\epsilon })_{0<\epsilon <\epsilon_{0}}$ bounded in $C\big([0,\underline{T}], X^{s+D}\big)$ and solving the initial value problems (\ref{alv})$_{0<\epsilon <\epsilon_0}$.
\end{theorem}
The spaces $\left\{ X^s \right\}_{s\geq 0}$ form a Banach scale and $F^s$ is defined as $F^s=C\big([0,T];X^s\big)\times X^{s+m}$. The constants $\delta $ and $P_{\min }$ appearing in the statement of the theorem are related to certain constants resulting from the above-mentioned three assumptions. For a detailed discussion of these assumptions, we refer the reader to \cite{Alvarez2008}.

In the rest of this section we will show that each of the  three assumptions of Theorem  \ref{theo4.1}, namely each of Assumptions 1.2, 1.3 and 1.5 of \cite{Alvarez2008}, also holds for our problem (\ref{systema})-(\ref{systemb}).  We first rewrite (\ref{systema})-(\ref{systemb}) in the form
\begin{equation}
    \mathbf{u}_{\tau }+{\frac{1}{\epsilon ^{p}}}\mathcal{L}\mathbf{u}+\mathcal{N}[\mathbf{u}]=\mathbf{0},~~~~
     \mathbf{u}(x,0) = \mathbf{u}_{0}(x), \label{sysa}
\end{equation}
where $\tau=\epsilon^{p}t$, $\mathbf{u}=(u,v)^{T}$, and the linear operator $\mathcal{L}$ and the nonlinear map $\mathcal{N}[.]$ are given respectively by
\begin{displaymath}
    \mathcal{L}=\left(
        \begin{array}{cc}
                0 & -KD_{x} \\
                -KD_{x} & 0%
        \end{array}%
        \right),\hspace*{1cm}\mathcal{N}\left[ \mathbf{u}\right] =\left(
        \begin{array}{c}
        0 \\
        -KD_{x}\big(u^{p+1}\big)%
        \end{array}%
        \right) .
\end{displaymath}
We note that our case is simpler than the one in \cite{Alvarez2008} since $\mathcal{L}$ and $\mathcal{N}$ are both independent of $\epsilon $ and $\tau$ so that $T$ below can be chosen arbitrarily large. Also,  the parameter $\epsilon $ appearing in (\ref{alv}) is replaced by $\epsilon ^{p}$ in our case.

For (\ref{sysa}) we work with the Banach space  $X^{s}$ and the smoothing operators   $S_{\theta }u=\mathcal{F}^{-1}\big(\chi_{[-\theta, \theta]}(\xi)\widehat{u}(\xi )\big)$ where $\chi$ is the characteristic function. These choices satisfy the requirements of a Banach scale.  We now proceed to show that  the three assumptions of \cite{Alvarez2008} also hold for (\ref{sysa}).

The first assumption  is about the  evolution operator generated by $\mathcal{L}$.
\begin{assumption}\label{asmp1}
    Consider the linear problem
    \begin{equation}
        \mathbf{u}_{\tau }+{\frac{1}{\epsilon ^{p}}}\mathcal{L}\mathbf{u}=\mathbf{0}, ~~~~\mathbf{u}(x,0)=\mathbf{g}(x).  \label{lin}
    \end{equation}
    \begin{enumerate}
        \item The linear operator $\mathcal{L}$ is uniformly bounded in $C\big(\mathbb{R};\mathfrak{L}(X^{s+1},X^{s})\big)$;
        \item The solution operator $U^{\epsilon }(\tau)$ for the linear problem (\ref{lin}) is uniformly bounded in $C\big(\mathbb{R};\mathfrak{L}(X^{s},X^{s})\big)$.
    \end{enumerate}
\end{assumption}
The following discussion shows that this assumption is valid.    $K$ is a bounded operator on $H^{s}$ due to (\ref{betabound}). Thus $KD_{x}$ maps $ H^{s+1}$ into $H^{s}$ and so $\mathcal{L}:X^{s+1}\rightarrow X^{s}$ is a bounded operator independent of $\tau $. The evolution operator $U^{\epsilon }$ satisfies $U^{\epsilon }(\tau )\mathbf{g}=\mathbf{u}(\tau)$. In Fourier space the solution of (\ref{lin}) is given by
\begin{displaymath}
        \left(\begin{array}{c}
            \widehat{u}(\xi ,\tau ) \\
            \widehat{v}(\xi ,\tau )
        \end{array}%
        \right) =e^{\tau \mathbf{A}}\left(
                                    \begin{array}{c}
                                        \widehat{g}_{1}(\xi ) \\
                                        \widehat{g}_{2}(\xi )%
                                    \end{array}
                                    \right)
\end{displaymath}
where
\begin{displaymath}
        \mathbf{A}=\left[
                    \begin{array}{cc}
                    0 & -i{\frac{\xi }{\epsilon ^{p}}}\sqrt{\widehat{\beta }(\xi )} \\
                    -i{\frac{\xi }{\epsilon ^{p}}}\sqrt{\widehat{\beta }(\xi )} & 0%
                    \end{array}
                    \right] .
\end{displaymath}%
The  operator $U^{\epsilon }(\tau )$ can be easily computed in the Fourier space as $U^{\epsilon }(\tau )=\mathcal{F}^{-1}e^{\tau \mathbf{A}}\mathcal{F}$ with
\begin{displaymath}
        e^{\tau \mathbf{A}}
            =\left[
            \begin{array}{cc}
            \cos \Big({\frac{\xi}{\epsilon^{p}}}\sqrt{\widehat{\beta}(\xi)}\tau\Big)
                & -i\sin \Big({\frac{\xi}{\epsilon^{p}}}\sqrt{\widehat{\beta}(\xi)}\tau \Big) \\
            -i\sin \Big({\frac{\xi}{\epsilon^{p}}}\sqrt{\widehat{\beta}(\xi)}\tau \Big)
                & \cos \Big({\frac{\xi }{\epsilon ^{p}}}\sqrt{\widehat{\beta }(\xi )}\tau \Big)
            \end{array}%
            \right] .
\end{displaymath}%
Obviously $U^{\epsilon }(\tau )$ is uniformly bounded in $C\big(\mathbb{R};\mathfrak{L}(X^{s},X^{s})\big)$.

The second assumption is about estimates for the nonlinear term.
\begin{assumption}\label{asmp2}
        For $s\geq s_{0}$ we have the following nonlinear estimates:
        \begin{eqnarray}
        && \hspace*{-35pt}\mbox{\emph{1.}} ~~~~
            \left\Vert \mathcal{N}[\mathbf{u}]\right\Vert _{X^s}
            \leq C\left\Vert \mathbf{u}\right\Vert_{X^{s_{0}}}^{p}\left\Vert \mathbf{u}\right\Vert _{X^{s+1}}; \label{as-21} \\
        && \hspace*{-35pt}\mbox{\emph{2.}} ~~~~
            \left\Vert \mathcal{N}_{\mathbf{u}}\left[ \mathbf{u}\right] \pmb{\phi}\right\Vert_{X^s}
            \leq C\Big(\left\Vert \mathbf{u}\right\Vert _{X^{s_{0}}}^{p}+\left\Vert \mathbf{u}\right\Vert _{X^{s_{0}}}^{p-1}\Big)\Big( \left\Vert \pmb{\phi}\right\Vert_{X^{s+1}}+\left\Vert \pmb{\phi}\right\Vert _{X^{s_{0}}}\left\Vert \mathbf{u}\right\Vert _{X^{s+1}}\Big); \label{as-22} \\
        && \hspace*{-35pt}\mbox{\emph{3.}} ~~~~
            \left\Vert\mathcal{N}_{\mathbf{uu}}\left[\mathbf{u}\right](\pmb{\phi},\pmb{\psi})\right\Vert_{X^s}
            \leq  C\Big(\left\Vert \mathbf{u}\right\Vert _{X^{s_{0}}}^{p-1}+\left\Vert \mathbf{u}\right\Vert _{X^{s_{0}}}^{p-2}\Big)\Big(\Vert \pmb{\phi}\Vert _{X^{s+1}}\Vert\pmb{\psi}\Vert _{X^{s_{0}}} \nonumber \\
        &&\hspace*{90pt}+\Vert \pmb{\phi}\Vert _{X^{s_{0}}}\Vert \pmb{\psi}\Vert _{X^{s+1}}
            +\Vert \mathbf{u}\Vert _{X^{s+1}}\Vert \pmb{\phi}\Vert _{X^{s_{0}}}\Vert\pmb{\psi}\Vert _{X^{s_{0}}}\Big). \label{as-23}
        \end{eqnarray}
\end{assumption}
We now show that the above assumption  is valid. By repeatedly applying (\ref{moser}) to $\mathcal{N}\left[ \mathbf{u}\right]=\big(0, -KD_{x}(u^{p+1})\big)$ where  $\mathbf{u}=(u,v)$, we get
        \begin{eqnarray*}
            \left\Vert \mathcal{N}[\mathbf{u}]\right\Vert _{X^s}
            &=&\left\Vert K D_{x}(u^{p+1})\right\Vert _{H^{s}}\leq C\left\Vert u^{p+1}\right\Vert _{H^{s+1}} \\
            &\leq & C\left\Vert u\right\Vert _{H^{s_{0}}}^{p}\left\Vert u\right\Vert_{H^{s+1}}
            \leq C\left\Vert \mathbf{u}\right\Vert _{X^{s_{0}}}^{p}\left\Vert \mathbf{u}\right\Vert _{X^{s+1}},
        \end{eqnarray*}
that is, (\ref{as-21}) holds. To check (\ref{as-22}) we note that $\mathcal{N}_{\mathbf{u}}\left[ \mathbf{u}\right] \pmb{\phi} = \big(0, -(p+1)KD_{x}(u^{p}\phi_{2})\big)$, where   $\pmb{\phi}=(\phi_{1},\phi_{2})$. Then we have
        \begin{eqnarray*}
            \left\Vert \mathcal{N}_{\mathbf{u}}\left[ \mathbf{u}\right] \pmb{\phi}\right\Vert_{X^{s}}
            &=&(p+1)\left\Vert KD_{x}(u^{p}\phi_{2})\right\Vert _{H^{s}}
                \leq C\left\Vert u^{p}\phi_{2}\right\Vert _{H^{s+1}} \\
            &\leq &  C\Big( \left\Vert u\right\Vert _{H^{s_{0}}}^{p}\left\Vert \phi_{2}\right\Vert_{H^{s+1}}+\left\Vert u\right\Vert _{H^{s_{0}}}^{p-1}\left\Vert \phi_{2}\right\Vert_{H^{s_{0}}}\left\Vert u\right\Vert _{H^{s+1}}\Big) \\
            &\leq & C \Big(\left\Vert \mathbf{u}\right\Vert _{X^{s_{0}}}^{p}+\left\Vert \mathbf{u}\right\Vert _{X^{s_{0}}}^{p-1}\Big)\Big( \left\Vert \pmb{\phi}\right\Vert_{X^{s+1}}+\left\Vert \pmb{\phi}\right\Vert _{X^{s_{0}}}\left\Vert \mathbf{u}\right\Vert _{X^{s+1}}\Big),
        \end{eqnarray*}
    that is, (\ref{as-22}) holds. Similarly, to check (\ref{as-23}) we first note that the only nonzero element of
    $\mathcal{N}_{\mathbf{uu}}\left[ \mathbf{u}\right] \left( \pmb{\phi},\pmb{\psi}\right)$ is
     $ -p(p+1)KD_{x}(u^{p-1}\phi_{2}\psi_{2})$  where $\pmb{\psi}=(\psi_{1},\psi_{2})$. Then we have
        \begin{eqnarray*}
            \left\Vert \mathcal{N}_{\mathbf{uu}}\left[ \mathbf{u}\right] (\pmb{\phi},\pmb{\psi})\right\Vert_{X^{s}}
            &\leq & p(p+1)\Vert K D_{x}(u^{p-1}\phi_{2}\psi_{2})\Vert _{H^{s}}
                \leq C\Vert u^{p-1}\phi_{2}\psi_{2}\Vert _{H^{s+1}} \\
            &\leq & C\Big(\Vert u\Vert _{H^{s_{0}}}^{p-1}\Vert \phi_{2}\Vert _{H^{s_{0}}}\Vert \psi_{2}\Vert _{H^{s+1}}
                +\Vert u\Vert _{H^{s_{0}}}^{p-1}\Vert \phi_{2}\Vert _{H^{s+1}}\Vert \psi_{2}\Vert _{H^{s_{0}}} \\
            && +\Vert u\Vert _{H^{s_{0}}}^{p-2}\Vert \phi_{2}\Vert_{H^{s_{0}}}\Vert \psi_{2}\Vert _{H^{s_{0}}}
                \Vert u\Vert _{H^{s+1}}\Big) \\
            &\leq &C\Big(\left\Vert \mathbf{u}\right\Vert _{X^{s_{0}}}^{p-1}
                +\left\Vert \mathbf{u}\right\Vert _{X^{s_{0}}}^{p-2})
                (\Vert \pmb{\phi}\Vert_{X^{s+1}}\Vert\pmb{\psi}\Vert _{X^{s_{0}}}
                +\Vert \pmb{\phi}\Vert _{X^{s_{0}}}\Vert \pmb{\psi}\Vert _{X^{s+1}} \\
            && +\Vert \mathbf{u}\Vert _{X^{s+1}}\Vert \pmb{\phi}\Vert _{X^{s_{0}}}
                \Vert\pmb{\psi}\Vert _{X^{s_{0}}}\Big),
        \end{eqnarray*}
that is, (\ref{as-23}) holds.

The third assumption is about the well posedness and a priori estimates for the linearized system:
\begin{assumption}\label{asmp3}
     Consider the linearized problem around $\overline{\mathbf{u}}=(\overline{u},\overline{v})$:
        \begin{equation}
            \mathbf{u}^{\epsilon}_{\tau }+{\frac{1}{\epsilon^{p}}}\mathcal{L}\mathbf{u}^{\epsilon}
                +\mathcal{N}_{\mathbf{u}}[\overline{\mathbf{u}}]\mathbf{u}^{\epsilon}
            =\mathbf{f}^{\epsilon },~~\mathbf{u}^{\epsilon}(x,0)
            =\mathbf{g}^{\epsilon }(x).  \label{linear}
        \end{equation}
     For given $\overline{\mathbf{u}}\in X_{(1)}^{s+1}$, and bounded families $\mathbf{f}^{\epsilon }\in X_{(0)}^{s}$, $\mathbf{g}^{\epsilon }\in X^{s}$ with $s\geq s_{0}+3$, the initial-value problem  (\ref{linear}) has a unique solution $\mathbf{u}^{\epsilon }\in X_{(0)}^{s}$ and for all $0< \epsilon < \epsilon_{0}$ and $\tau \in [0, T]$
    \begin{displaymath}
        \left\Vert \mathbf{u}^{\epsilon}\right\Vert _{X_{(0)}^{s}}
        \leq C\big(T,\left\Vert \overline{\mathbf{u}}\right\Vert _{X_{(1)}^{s_{0}+2}}\big)
            \Big(\mathfrak{I}^{s}(\tau,\mathbf{f}^{\epsilon },\mathbf{g}^{\epsilon})
            +\left\Vert\overline{\mathbf{u}}\right\Vert_{X_{(1)}^{s+1}}\mathfrak{I}^{s_{0}+1}(\tau,\mathbf{f}^{\epsilon }\mathbf{,g}^{\epsilon })\Big).
    \end{displaymath}
\end{assumption}
In the remaining part of this section we show that the above assumption holds for our system.

We first note that  (\ref{linear}) is just (\ref{l1tau})-(\ref{l2tau}) with $w=(p+1)\overline{u}^{p}$. For the existence proof of the linearized system, we follow the standard hyperbolic approach \cite{Taylor2011} to obtain a solution of (\ref{l1tau})-(\ref{l2tau}) as a limit of solutions of the regularized system
\begin{equation}
    \mathbf{u}_{\tau }^{h}+{\frac{1}{\epsilon^{p}}}J^{h}\mathcal{L}\mathbf{u}^{h}+J^{h}\mathcal{N}_{\mathbf{u}}
        [\overline{\mathbf{u}}]\mathbf{u}^{h}
    =\mathbf{f},~~\mathbf{u}^{h}(x,0)=\mathbf{g}(x).  \label{mol}
\end{equation}
where we omit the index $\epsilon$ for convenience. In (\ref{mol}) $J^{h}$ is the Friedrichs mollifier given by
\begin{displaymath}
    J^{h}z(x)=\frac{1}{h}\int \eta \Big(\frac{x-y}{h}\Big)z(y)dy
\end{displaymath}
with some nonnegative $\eta \in C_{0}^{\infty }(\mathbb{R)}$ with $\int \eta dx=1$.

Due to regularizing effect of $J^{h},$ (\ref{mol}) is a system of  $X^{s}$-valued ODEs and hence has solution $\mathbf{u}^{h}$ in $X_{(0)}^{s}$. Moreover, since $J^{h}\mathcal{L}$ and $J^{h}\mathcal{N}_{\mathbf{u}}$ satisfy the same bounds as $\mathcal{L}$ and $\mathcal{N}_{\mathbf{u}}$, respectively; Corollary \ref{cor3.4} will hold uniformly in $h$. Using the inequality
\begin{displaymath}
    \left\Vert w\right\Vert _{H_{(1)}^{s+1}}
    =(p+1)\left\Vert \overline{u}^{p}\right\Vert _{H_{(1)}^{s+1}}\leq C \left\Vert \overline{u}\right\Vert_{H_{(1)}^{s_{0}}}^{p-1}\left\Vert \overline{u}\right\Vert _{H_{(1)}^{s+1}},
\end{displaymath}
we then have, for all $\epsilon < \epsilon _{0}$,
\begin{eqnarray*}
    \left\Vert \mathbf{u}^{h}(\tau)\right\Vert _{X^{s}}
    &\leq & C\big(T,\left\Vert w\right\Vert_{H_{(1)}^{s_{0}+2}}\big)\Big(\mathfrak{I}^{s}(\tau ,\mathbf{f,g})
        +    \left\Vert w\right\Vert_{H_{(1)}^{s+1}}\mathfrak{I}^{s_{0}+1}(\tau ,\mathbf{f,g})\Big) \\
    &\leq &C\big(T,\left\Vert \overline{u}\right\Vert _{H_{(1)}^{s_{0}+2}}\big)\Big(\mathfrak{I}^{s}(\tau ,\mathbf{f,g})
        +\left\Vert \overline{u}\right\Vert_{H_{(1)}^{s+1}}\mathfrak{I}^{s_{0}+1}(\tau, \mathbf{f,g})\Big).
\end{eqnarray*}%
It follows that $\left\Vert \mathbf{u}^{h}\right\Vert _{X_{(0)}^{s}}$ is uniformly bounded on $\left[ 0,T\right] $.

We now show that for any sequence $\big(h_{n}\big)$ with $h_{n}\rightarrow 0$, the solutions $\mathbf{u}^{h_{n}}$ form a Cauchy sequence in $X_{(0)}^{s-2}$ so that $\lim \mathbf{u}^{h_{n}}=\mathbf{u}$ exists in $X_{(0)}^{s-2}$.  We proceed as follows: We fix $\epsilon $ and let $\mathbf{r=}$ $\mathbf{u}^{h_{n}}-\mathbf{u}^{h_{m}}.$ Then $\mathbf{r}$ satisfies
\begin{displaymath}
    \mathbf{r}_{\tau }+{\frac{1}{\epsilon^{p}}}J^{h_{n}}\mathcal{L}\mathbf{r}
        +J^{h_{n}}\mathcal{N}_{\mathbf{u}}[\overline{\mathbf{u}}]\mathbf{r}
    =\big(J^{h_{m}}-J^{h_{n}}\big)\Big(\frac{1}{\epsilon^{p}}\mathcal{L}\mathbf{u}^{h_{m}}
        +\mathcal{N}_{\mathbf{u}}[\overline{\mathbf{u}}]\mathbf{u}^{h_{m}}\Big),~~~~
    \mathbf{r}(x,0)=\mathbf{0}.
\end{displaymath}
This equation is of the form (\ref{l1tau})-(\ref{l2tau}) with the nonhomogeneous term $\widetilde{\mathbf{f}}=\big(J^{h_{m}}-J^{h_{n}}\big)
\big(\frac{1}{\epsilon^{p}}\mathcal{L}\mathbf{u}^{h_{m}}+\mathcal{N}_{\mathbf{u}}
[\overline{\mathbf{u}}]\mathbf{u}^{h_{m}}\big)$ and $\mathbf{g}=\mathbf{0}$.
 Using the mollifier estimate \cite{Mats2017}
\begin{displaymath}
    \left\Vert J^{h_{1}}z-J^{h_{2}}z\right\Vert _{H^{s}}
    \leq \left\vert h_{1}-h_{2}\right\vert \left\Vert z\right\Vert _{H^{s+1}},
\end{displaymath}
we have
\begin{displaymath}
    \left\Vert \widetilde{\mathbf{f}}\right\Vert _{X_{(0)}^{s-2}}
    \leq \left\vert h_{n}-h_{m}\right\vert \left\Vert \frac{1}{\epsilon ^{p}}\mathcal{L}\mathbf{u}^{h_{m}}+\mathcal{N}_{\mathbf{u}}[\overline{\mathbf{u}}]\mathbf{u}^{h_{m}}\right\Vert _{X_{(0)}^{s-1}}
    \leq C\left\vert h_{n}-h_{m}\right\vert \left\Vert \mathbf{u}^{h_{m}}\right\Vert _{X_{(0)}^{s}},
\end{displaymath}
where we have used the fact that the operator $\mathcal{L}$ $\ $maps $X^{s} $ into $X^{s-1}$.  As $\left\Vert \mathbf{u}^{h_{m}}\right\Vert _{X_{(0)}^{s}}$ is uniformly bounded, this yields $\left\Vert \widetilde{\mathbf{f}}\right\Vert _{X_{(0)}^{s-2}}\leq C\left\vert h_{n}-h_{m}\right\vert $. For $s\geq s_{0}+3$ this implies that $ \mathfrak{I}^{s-2}(\tau,\widetilde{\mathbf{f}},\mathbf{0})$ and $ \mathfrak{I}^{s_{0}+1}(\tau,\widetilde{\mathbf{f}},\mathbf{0})$  are both bounded by $C \vert h_{n}-h_{m}\vert$. Then  Corollary \ref{cor3.4} gives
\begin{displaymath}
    \left\Vert \mathbf{u}^{h_{n}}-\mathbf{u}^{h_{m}}\right\Vert _{X_{(0)}^{s-2}}\leq C\left\vert h_{n}-h_{m}\right\vert .
\end{displaymath}
Thus $\big(\mathbf{u}^{h_{n}}\big)$ is a Cauchy sequence  in $X_{(0)}^{s-2}$. Moreover,  it follows from (\ref{l1tau})-(\ref{l2tau}) that $(\mathbf{u}_{\tau }^{h_{n}})$ is also a Cauchy sequence in $X_{(0)}^{s-3}$; in other words we get that ($\mathbf{u}^{h_{n}})$ is Cauchy in $X_{(1)}^{s-2}$.  It is clear that $\lim \mathbf{u}^{h_{n}}=\mathbf{u\in }X_{(1)}^{s-2}$ and $\mathbf{u}$ satisfies (\ref{linear}). Next we show that $\mathbf{u}$ is in $X^{s}_{(1)}$. We first observe that $\left\Vert \mathbf{u}^{h_{n}}\right\Vert_{X^{s}_{(0)}}$ is bounded and thus $(\mathbf{u}^{h_{n}})$ has a weakly convergent subsequence in $X^{s}_{(0)}$. By uniqueness of the limit, this weak limit should be $\mathbf{u}$ which gives the regularity result $\mathbf{u}\in X^{s}_{(0)}$.  Corollary \ref{cor3.4} then gives the required estimate for $\left\Vert \mathbf{u}(\tau)\right\Vert _{X^{s}}$. Since (\ref{linear}) is linear, uniqueness of the solution follows again from Corollary \ref{cor3.4}.

Thus we have established that the three assumptions  of Theorem \ref{theo4.1}  hold for our problem (\ref{sysa}).

\section{Long-time existence of solutions}\label{sec5}

We are now ready to apply Theorem \ref{theo4.1}  to our problem (\ref{sysa}). We  first discuss the parameters appearing in the statement of the theorem. Recall that  $s_0>{1\over 2}$  in our case. As mentioned in Section \ref{sec4}, $T>0$ can be taken arbitrarily large.  The other three parameters $m$, $d_{1}, d_{1}^{\prime}$ appearing in Theorem \ref{theo4.1} refer to  loss of derivatives in the assumptions. We note that, although different values of $m$ appear in the proofs above,  all the assumptions will hold for $m=3$. As a result we have   $m=3$, $d_{1}=1$ and $d_{1}^{\prime}=0$ in our case. The quantities $\delta$, $q$ and $P_{\min}$ are defined as
\begin{equation}
    \delta =\max \{d_{1},d_{1}^{\prime}+m\}, ~~~~    q =D-m-d_{1}^{\prime }, ~~~~
    P_{\min} =\delta +\frac{D}{q}\big (\sqrt{\delta }+\sqrt{2(\delta +q)}\big)^{2}
\end{equation}
in \cite{Alvarez2008}. Thus, these quantities take the values
\begin{equation}
    \delta =3, ~~~~    q =D-3, ~~~~     P_{\min} =3+{\frac{D}{D-3}}\big(\sqrt{3}+\sqrt{2D}\big)^{2}.
\end{equation}
In Theorem \ref{theo4.1} all assumptions are supposed to hold for $\epsilon < \epsilon_{0}$. In our case, for a given
$w$, the parameter $\epsilon_{0}$  is determined by Lemma \ref{lem3.1}. Considering that  $w$ is replaced by $(p+1)u^{p}$ in the linearization (\ref{l2tau}), we see that $\epsilon_{0}$ is determined by the initial data. Thus  we obtain the local well-posedness of   (\ref{sysa}) as follows:
\begin{corollary}\label{maintau}
    Let $D>3$, $P>P_{\min}$, $~s\geq s_{0}+3$ and $\big(\mathbf{u}_{0}^{\epsilon}\big)$ be bounded in $X^{s+P}$. Then there exist some $\epsilon_0>0$, $~T>0$ and a unique family  $\big(\mathbf{u}^{\epsilon}\big)_{0<\epsilon<\epsilon_0}$ bounded in $C\big([0,T];X^{s+D}\big)$ and satisfying (\ref{sysa}) with initial values  $\mathbf{u}_{0}(x) =\mathbf{u}^{\epsilon}_{0}(x)$.
\end{corollary}
We are now in a position to state our main results.  Corollary \ref{maintau} provides the local well-posedness of solutions to (\ref{sysa}) on the interval $[0, T]$ of the scaled time variable $\tau$. Setting $t=\tau/\epsilon^p$, that is, changing back from the scaled time variable $\tau$ to the original time variable $t$ we get  the long-time existence result for solutions of the initial value problem (\ref{systema})-(\ref{systemb}):
\begin{theorem}\label{realt}
    Let $D>3$, $P>P_{min}$, $~s\geq s_{0}+3$ and $\big(\mathbf{u}_{0}^{\epsilon}\big)$ be bounded in $X^{s+P}$. Then there exist some $\epsilon_0 >0$, $~T>0$ and a unique family  $\big(\mathbf{u}^{\epsilon}\big)_{0<\epsilon<\epsilon_0}$ bounded in $C\big([0,{T\over \epsilon^p}]; X^{s+D}\big)$ and satisfying (\ref{systema})-(\ref{systemb}) with initial values  $\mathbf{u}_{0}(x) =\mathbf{u}^{\epsilon}_0(x)$.
\end{theorem}
This is our first main result and it allows us to say that  solutions to (\ref{systema})-(\ref{systemb}) exist over  the long time  interval $[0,{T\over \epsilon^p}]$ of $t$. Now we want to state  an analog of the above theorem  for solutions of the initial-value problem (\ref{nw})-(\ref{ini}). The basic issue is how to choose the initial data $(u_{0}, v_{0})$ for  (\ref{systema})-(\ref{systemb}) when the initial data $(u_{0}, u_{1})$ for (\ref{nw})-(\ref{ini}) is given.  Obviously,  $u_{0}$'s are the same and we have $u_{1}= K\big(v_0\big)_x$. Hence if we have the initial data in the form $u_{1}=(w_{0})_{x}$, we get $w_{0}=Kv_{0}$. When $K$ is invertible, we have $v_{0}=K^{-1}w_{0}$. We note that invertibility of $K$ is equivalent to the strict positivity (ellipticity) condition
\begin{equation}
    0<C_{1} \leq \widehat{\beta}(\xi)\leq C_{2} \label{betapos}
\end{equation}
on the kernel. In this case, we have:
\begin{theorem}\label{realw}
    Suppose the kernel $\beta$ satisfies the ellipticity condition (\ref{betapos}). Let $D>3$, $P>P_{min}$, $~s\geq s_{0}+3$ and $\big(u_0^{\epsilon}, w_{0}^{\epsilon}\big)$ be bounded in $H^{s+P}\times H^{s+P}$. Then there exist some $\epsilon_0>0$, $~T>0$ and a unique family  $\big(u^{\epsilon} \big)_{0<\epsilon<\epsilon_0}$ bounded in $C\big([0,{T\over \epsilon^p}];H^{s+D}\big)\times C^{1}\big([0,{T\over \epsilon^p}];H^{s+D-1}\big)$ and satisfying (\ref{nw})-(\ref{ini}) with initial values $u_{0}(x)=u_{0}^{\epsilon}(x)$,  $u_{1}(x)=\big(w_0^{\epsilon}(x)\big)_x$.
\end{theorem}
\begin{remark}\label{latrem}
    If $\beta$ is not elliptic, then $K$ is not invertible and we cannot perform the above transformation. In this case, if $u_{1}$ is of the form $u_{1}=\big(Kv_{0}(x)\big)_{x}$ with $\big(u_0^{\epsilon}, v_{0}^{\epsilon}\big)$ bounded in $H^{s+P}\times H^{s+P}$, then Theorem \ref{realw} will still be valid.
\end{remark}

If we make the transformation  $U=\epsilon u$, the Cauchy problem (\ref{nw})-(\ref{ini}) takes the form
\begin{eqnarray}
    U_{tt}=\beta\ast \big( U+ U^{p+1}\big)_{xx}, ~~~~ x\in \mathbb{R},~~~t>0,  \label{nw-u} \\
    U(x,0)=\epsilon u_{0}(x), ~~~~ U_{t}(x,0)=\epsilon u_{1}(x),  \label{ini-u}
\end{eqnarray}
 where the small parameter $\epsilon$ is transferred from the nonlocal equation to the initial conditions. Then we get the following result about the long-time existence for small initial data.
\begin{theorem} \label{bigU}
    Suppose $\beta$ satisfies the ellipticity condition (\ref{betapos}). Let $D>3$, $P>P_{min}$, $~s\geq s_{0}+3$ and $\big(u_0, w_{0}\big)$ be bounded in $H^{s+P}\times H^{s+P}$. Then there exist some $\epsilon_0>0$, $~T>0$ and a unique family  $\big(U^{\epsilon}\big)_{0<\epsilon<\epsilon_0}$ bounded in $C\big([0,{T\over \epsilon^p}]; H^{s+D}\big)\times C^{1}\big([0,{T\over \epsilon^p}];H^{s+D-1}\big)$ and satisfying (\ref{nw-u})-(\ref{ini-u}) with $u_{1}=(w_0)_x$.
\end{theorem}
\begin{remark}
    The extra smoothness requirement $\big(\mathbf{u}_{0}^{\epsilon}\big)\in X^{s+P}$ (whereas solution is in $X^{s+D}$ with $D<P$) in the above theorems is due to the technical aspect of the Nash-Moser scheme. Computation shows that  the optimal values are obtained when the pair $(P, D)$ is  approximately $(55.34, 7.35)$.  In certain cases the lower bound for $P$ may be improved. In particular, we will consider the case when $\beta$  is more regular, namely when (\ref{order}) is satisfied with $r\geq 2$. In this case both $K$ and $KD_x$ are bounded operators on $H^s$. Then (\ref{sysa}) is an initial-value problem defined for $H^{s}$-valued ODEs, so that we have local existence without loss of derivatives. Moreover, with the simplified energy
    \begin{displaymath}
        \widetilde{E}_{s}^{2}(\tau)=\frac{1}{2}\Big( \Vert u^{\epsilon}(\tau)\Vert_{H^{s}}^{2}
                                +\Vert v^{\epsilon}(\tau)\Vert_{H^{s}}^{2}\Big),
    \end{displaymath}
    one directly gets from (\ref{sysa})
    \begin{displaymath}
        {d\over {d\tau}}\widetilde{E}_{s}^{2}(\tau)=\big\langle KD_{x}\big(u^{\epsilon}(\tau)\big)^{p+1}, v^{\epsilon}(\tau)\big\rangle_{H^{s}}\leq C \widetilde{E}_{s}^{p+2}(\tau).
    \end{displaymath}
    This differential inequality  gives  uniform estimates for the solutions $\big(u^{\epsilon}(\tau),  v^{\epsilon}(\tau)\big)$. Changing back the variable to $t$, we get the following result for (\ref{systema})-(\ref{systemb}) with no extra smoothness requirement. Clearly the same conclusion will be valid for the other theorems, namely when $r\geq 2$ we can take $P=D=0$. We note that this result is the long-time version of the local existence result given in \cite{Duruk2010}.
\end{remark}
\begin{theorem}\label{theoD}
    Let $s\geq s_{0}$. Suppose the kernel $\beta$ satisfies (\ref{order}) with $r\geq 2$. Let $\big(\mathbf{u}_{0}^{\epsilon}\big)$ be bounded in $X^{s}$. Then there exist some $\epsilon_0 >0$, $~T>0$ and a unique family  $\big(\mathbf{u}^{\epsilon}\big)_{0<\epsilon<\epsilon_0}$ bounded in $C^{1}\big([0,{T\over \epsilon^p}]; X^{s}\big)$ and satisfying (\ref{systema})-(\ref{systemb}) with initial values  $\mathbf{u}_{0}(x) =\mathbf{u}^{\epsilon}_0(x)$.
\end{theorem}

In the remainder of this section, we  provide some examples of the general class (\ref{nw}): the classical elasticity  equation, the improved Boussinesq equation and the lattice equation.
\begin{example}\label{exam1}(The Classical  Elasticity Equation)
     When $\beta=\delta$ where $\delta$ is the Dirac measure, we capture the well-posedness result for the one-dimensional classical (local) elasticity equation  (\ref{wa-eq}).    Here $\widehat{\beta }=1$ so $K$ is the identity operator and Theorem \ref{realw} holds.
\end{example}
\begin{example}\label{exam2}(The Improved Boussinesq Equation)
    When $\beta (x)=\frac{1}{2}e^{-\left\vert x\right\vert }$, the nonlocal equation (\ref{nw}) becomes the improved Boussinesq equation given by (\ref{ib}). The Fourier transform of the kernel is  $\widehat{\beta }(\xi )=(1+\xi^{2})^{-1}$. So the kernel satisfies both the regularity assumption (\ref{order}) with $r=2$ and  the ellipticity condition (\ref{betapos}).  Then Theorem \ref{theoD} holds.
\end{example}
\begin{example}\label{exam3}(The Lattice Equation)
    The class (\ref{nw}) involves  also the differential-difference equation
    \begin{equation}
        u_{tt}=\Delta^{d}\big(u+\epsilon ^{p}u^{p+1}\big)  \label{lattice}
    \end{equation}%
    with the discrete Laplacian operator
    \begin{equation}
        \Delta^{d}z=z(x-1,t)-2z(x,t)+z(x+1,t).  \label{lap}
    \end{equation}%
    Equation (\ref{lattice}) is widely used as a model in  one dimensional non-linear lattice dynamics. Indeed for the triangular kernel,   $\beta (x)=1-|x|$ for $|x|\leq 1 $ and $\beta(x)=0$ elsewhere, (\ref{nw}) reduces to (\ref{lattice}). We have  $\widehat{\beta }(\xi )={4\over {\xi^{2}}}\sin^{2}({\xi\over 2})$ which satisfies (\ref{order}) with $r=2$ but fails to be elliptic. The operator $K$ with symbol ${2\over \xi}\sin({\xi\over 2})$ can be explicitly calculated, $Ku=\chi_{[{{-1}\over 2}, {1\over 2}]} \ast u$ with the characteristic function  $\chi$. Then following Remark \ref{latrem}, for initial data of the form   $\big(u_0^{\epsilon}, u_{1}^{\epsilon}\big)$ with $u_{1}^{\epsilon}=\big(Kv_{0}^{\epsilon}(x)\big)_{x}$ and $\big(u_0^{\epsilon}, v_{0}^{\epsilon}\big)$ bounded in $H^{s}\times H^{s}$, we have uniform bounds and long time existence in $C^{1}\big([0,{T\over \epsilon^p}]; H^{s}\big)$.
\end{example}
Local existence of smooth solutions of (\ref{wa-eq})  is a well-known result. In fact the much more complicated 3D case can be found in \cite{Hughes1977}. For long-time existence of the 3D case we refer to \cite{Klainerman1984} and the references therein. Our result in Example \ref{exam1} recaptures in fact  the simpler 1D case. The wave equations in Examples \ref{exam2} and \ref{exam3} are particular cases of (\ref{nw}) with $r\geq 2$ and the local well-posedness for those equations follows from the general result in \cite{Duruk2010}. Our results here yield long-time existence and uniform estimates of the solutions.

\end{document}